\documentclass{amsart}
\usepackage{amssymb}
\usepackage{amsmath}

\newtheorem{Theorem}{Theorem}

\newtheorem{Lemma}[Theorem]{Lemma}

\newtheorem{Proposition}[Theorem]{Proposition}

\newtheorem{Corollary}[Theorem]{Corollary}

\newtheorem{Remark}[Theorem]{Remark}

\newtheorem{Definition}[Theorem]{Definition}

\title[Bohr Property]{Bohr property of bases in the space of entire
functions and its generalizations}

\author{Aydin Aytuna}

\address{Sabanci University, Orhanli,
34956 Tuzla, Istanbul, Turkey}
 \email{aytuna@sabanciuniv.edu}

\author{Plamen Djakov}

\begin{document}
\address{Sabanci University, Orhanli,
34956 Tuzla, Istanbul, Turkey}
 \email{djakov@sabanciuniv.edu}



\dedicatory{Dedicated to Tosun Terzio\u{g}lu on the occasion of his
seventieth birthday}

\begin{abstract}
We prove that if $(\varphi_n)_{n=0}^\infty, \; \varphi_0 \equiv 1, $
is a basis in the space of entire functions of $d$ complex variables,
$d\geq 1,$  then for every compact $K\subset \mathbb{C}^d$ there is a
compact $K_1 \supset K$ such that for every entire function $f=
\sum_{n=0}^\infty f_n \varphi_n$ we have $\sum_{n=0}^\infty |f_n|\,
\sup_{K}|\varphi_n| \leq \sup_{K_1} |f|.$ A similar assertion holds
for bases in the space of global analytic functions on a Stein
manifold with the Liouville Property.
\end{abstract}

\maketitle

\section{Introduction}
In 1914 H. Bohr \cite{B}  proved that if $f= \sum c_{n}z^{n}$ is a
bounded analytic function  on the unit disc $U \subset \mathbb{C},$
then
$$
\sum_{n=0}^\infty | c_{n}| r^n  \leq \sup_{z\in U} |f(z)|
$$
for every  $0 \leq r \leq \frac{1}{6}.$  The largest $r\leq 1$ such
that the above inequality holds is referred to as the {\em Bohr
radius}, $\kappa_1, $ for the unit disc. The exact value of
$\kappa_1$ was computed, by M. Riesz, I. Schur and N. Wiener, to be
$\frac{1}{3}.$

In 1997 H. P. Boas and D. Khavinson \cite{BK} showed that a similar
phenomenon occurs for polydiscs in $\mathbb{C}^d.$ If we let $U^{d}$
denote the unit polydisc in $\mathbb{C}^{d},$ the largest number $r$
such that if $\left | \sum_{\alpha } c_{\alpha } z^{\alpha } \right |
<1$ for all\ $z \in U^{d},$ then $\sum_{\alpha } \left | c_{\alpha
}z^{\alpha }\right | <1$ holds for the homothetic domain $ r\,U^{d},$
is referred to as the {\em Bohr radius}, $\kappa_{d},$ for the unit
polydisc $U^{d}.$ Boas and Khavinson obtained upper and lower bounds
for $\kappa_{d},$ in terms of $d,$ and showed that $\kappa_{d}
\rightarrow 0$ as $ d \rightarrow \infty .$ However the {\em exact}
value of $\kappa_{d} $ is still not known. Recently  A. Defant, L.
Frerick, J. Ortega-Cerd\`{a}, M. Ouna\"{\i}es and K. Seip
\cite{DFOOS} showed that $\kappa_{d}$ behaves asymptotically as
$\sqrt{\frac{\log d}{d}},$ modulo a factor bounded away from zero and
infinity. Other multidimensional variants of Bohr's phenomenon were
given by L. Aizenberg~\cite{Aiz00}. He proved Bohr-type theorems for
bounded complete Reinhardt domains and obtained estimates for the
corresponding Bohr radii.

P. G. Dixon \cite{Dix}  has used Bohr's original theorem to construct
a Banach algebra which is not an operator algebra, yet satisfies the
non-unital von Neumann's inequality. V. Paulsen, G. Popescu and D.
Singh \cite{PPS} have applied operator-theoretic techniques to obtain
refinements and multidimensional generalizations of Bohr's
inequality.

Interesting interconnections among multidimensional Bohr radii, local
Banach space theory and complex analysis in infinite number of
variables established in \cite{DGM} and \cite{DT} triggered a further
wave of investigations. For this line of research and recent related
references we refer the reader to the survey \cite{DP}.

Ramifications and extensions of Bohr-type theorems also attracted
attention. Various authors studied versions of Bohr phenomena in
different settings. See for example \cite{Aiz05}, \cite{BB},
\cite{DR}, \cite{G}.

In \cite{AAD00} and \cite{AAD01} we (along with L. Aizenberg) took a
more abstract approach to Bohr phenomena and considered general bases
in the space of global analytic functions on complex manifolds rather
than monomials. For a complex manifold $\mathcal{M},$  a given basis
$\left ( \varphi_{n} \right)_{n=0}^\infty $ in the space
$H(\mathcal{M})$ of global analytic functions is said to have the
{\em Bohr Property} $\left( BP\right) $ if there exist an open set $U
\subset \mathcal{M}$ and a compact set $K \subset \mathcal{M}$ such
that, for every  $f=\sum c_{n}\varphi_{n}$ in $H(\mathcal{M}),$ the
inequality
\begin{equation}
\label{0} \sum | c_{n} | \sup_{U} |\varphi _{n} (z)| \leq \sup_{K}
|f(z)|
\end{equation}
 is valid. In \cite[Theorem 3.3]{AAD01} it is shown that a
basis $\left( \varphi _{n}\right) _{n=0}^{\infty }$ has $BP$
 if $\varphi _{0}=1$ and there is a point
$z_0 \in \mathcal{M}$ such that $\varphi_{n} (z_0) =0,$ $n=1,2,
\ldots. $

Let us note that Theorem 3.3 in \cite{AAD01} has a {\em local
character}, namely in fact it proves that for any compact
neighborhood $K\ni z_0$ there is an open set $U$ with $z_0 \in U
\subset K $ such that (\ref{0}) holds. Moreover, its proof is based
on considering sets $U$ that "shrink" to $z_0.$

Recently P. Lass\`{e}re and E. Mazzilli \cite{LM1} (see also
\cite{KS} and \cite{LM2}) have studied the Bohr phenomenon for the
Faber polynomial basis $\left( \psi_{n}\right)_{n=0}^\infty $
associated to a continuum in $\mathbb{C}.$  By using Theorem 3.3 in
\cite{AAD01} and some properties of Faber polynomials they proved
that for {\em every } relatively compact $ W \subset \mathbb{C}$
there is a compact $K \subset \mathbb{C}$ such that, for every entire
function $f=\sum c_{n}\psi_{n},$
$$
\sum | c_{n}| \sup_{W} |\psi_{n} (z)| \leq \sup_{K} |f(z)|.
$$
Let us note that the latter assertion has a {\em global character}.

In this paper we give a characterization of the bases that possess
{\em global } $BP$  in the above sense, for a class of complex
manifolds which contains $\mathbb{C}^{d}, 1\leq d<\infty $ and more
generally parabolic Stein manifolds (see \cite{AKT,AS}).  Our results
extend and generalize the above mentioned theorem of P. Lass\`{e}re
and E. Mazzilli  \cite[Theorem 3.1]{LM1}. See more comments on their
results in Section~3, after Remark~\ref{rem4}.

We recall some basic definitions and facts and get preliminary
results in Section 2. In Section~3 we prove that the Global Bohr
Property takes place for every basis $(\varphi_n)_{n=0}^\infty $ in
the space of entire functions $H(\mathbb{C}^d)$ provided one of the
functions $\varphi_n$  is a constant. In Section~4 we generalize the
results from Section~3 for Stein manifolds with Liouville Property,
i.e., for manifolds with the property that every bounded analytic
function reduces to a constant.

\section{Preliminaries}
Let $D \subset \mathbb{C}^d$  be a domain in $\mathbb{C}^d$ (or, $D
\subset \mathcal{M},$ where $\mathcal{M}$ is a Stein manifold). We
denote by $H(D)$ the space of analytic functions on $D.$ For any
compact subset $K \subset D $ we set
$$ |f|_{K} \;: = \; \sup_{K} |f(z)| , \quad f \in H(D). $$
Further we write $K \subset \subset D$ if $K$ is a compact subset of
$D.$  The system of seminorms $|f|_{K}, \; K \subset \subset D, $
defines the topology of uniform convergence on compact subsets of
$D.$ Regarded with it $H(D)$ is a nuclear Fr\'echet space (e.g.
\cite{MV}).

Recall that a sequence $(\varphi_n )_{n=0}^{\infty} $ of analytic
functions on $D $ is called {\em basis} in $H(D) $ if for every
function $f \in H(D)$ there exists a unique sequence of complex
numbers $f_n $ such that
$$f \; = \; \sum_{n=0}^\infty  f_n \varphi_n , $$
where the series converges uniformly on any compact subset of $D.$

We will use the following well known fact.

\begin{Proposition}
\label{abs} If $(\varphi_n  )_{n=0}^{\infty} $ is a basis in  $H(D)$
then for every $K_1 \subset \subset D $  there exist $K_2 \subset
\subset D $ and $C>0 $ such that
$$\text{if} \;\; f  =  \sum f_n \varphi_n \quad \text{then} \quad
  \sum |f_n | |\varphi_n |_{K_1} \;
\leq \;  C|f|_{K_2} .$$
\end{Proposition}

The assertion follows from the theorem of Dynin and Mityagin on
absoluteness of bases in nuclear spaces (\cite{DM60,M}, e.g.
\cite{MV}, Theorem 28.12).

\begin{Lemma}
\label{lem1} Let $(\varphi_n  )_{n=0}^{\infty} $ be a basis in
$H(\mathbb{C}^d), \; d\geq 1,$ such that $\varphi_0 (z) \equiv 1.$
Then for every compact $K \subset \mathbb{C}^d$ we have
\begin{equation}
\label{11} \inf_{K_1 \subset \subset \mathbb{C}^d } \sum_{n=1}^\infty
\frac{|\varphi_n|_K}{|\varphi_n|_{K_1}} =0.
\end{equation}
\end{Lemma}

\begin{proof}
Since $\varphi_0 (z) \equiv 1,$  we have that $\varphi_n \neq const $
for $n\in \mathbb{N}.$ Therefore, each of the function $\varphi_n, \;
n \in \mathbb{N} $ is unbounded. This implies that
\begin{equation}
\label{12} \inf_{K_1 \subset \subset \mathbb{C}^d }
\frac{|\varphi_n|_K}{|\varphi_n|_{K_1}} =0, \quad \forall n \in
\mathbb{N}.
\end{equation}
On the other hand, by Grothendieck-Pietsch criterion for nuclearity
(e.g. \cite{MV}, Theorem 28.15) we have
\begin{equation}
\label{13} \forall \, K \subset \subset \mathbb{C}^d \quad \exists \,
K_1 \subset \subset \mathbb{C}^d  \; : \quad \sum_0^\infty
{|\varphi_n|_{K}}/{|\varphi_n|_{K_1}} \; < \, \infty.
\end{equation}
Now (\ref{11}) follows immediately from (\ref{12}) and (\ref{13}).

\end{proof}

\section{Bohr property of bases in the space of entire functions}

 We say that a basis $(\varphi_n )_{n=0}^{\infty} $ in the
space of entire functions $H(\mathbb{C}^d)$ has the {\em Global Bohr
Property} ($GBP$) if for every compact $K  \subset \mathbb{C}^d$
there is a compact $K_1 \supset K$ such that
\begin{equation}
\label{e2} \text{if} \;\;  f=\sum_0^\infty f_n \varphi_n, \;\;
\text{then} \;\; \sum_0^\infty |f_n | |\varphi_n |_K \; \leq \;
|f|_{K_1} \quad \forall f \in H(\mathbb{C}^d).
\end{equation}

\begin{Theorem}
\label{thm1}  A basis $(\varphi_n)_{n=0}^\infty$  in the space of
entire functions $H(\mathbb{C}^d)$  has $GBP$ if and only if one of
the functions $\varphi_n$ is a constant.
\end{Theorem}

\begin{proof}
Let $(\varphi_n)_{n=0}^\infty$ be a basis in  $H(\mathbb{C}^d)$ which
has $GBP$. If $ 1= \sum_{n=0}^\infty c_n \varphi_n (z) $ is the
expansion of the constant function $1,$ then at least one of the
coefficients $c_n $ is nonzero, say $c_{n_0} \neq 0. $ By (\ref{e2}),
it follows that for every $K\subset \subset \mathbb{C}^d$ there is
$K_1 \subset \subset \mathbb{C}^d$ such that
$$
|c_{n_0}| |\varphi_{n_0}|_K \leq  |1|_{K_1}=1.
$$
Hence $|\varphi_{n_0} (z)| \leq 1/|c_{n_0}|, $ i.e., $\varphi_{n_0}
(z)$ is a bounded entire function, so it is a constant by the
Liouville theorem. (The necessity assertion follows also from the
argument that proves Proposition~3.1 in \cite{AAD01}.)

Suppose that  $(\varphi_n )_{n=0}^{\infty} $ is a basis in
$H(\mathbb{C}^d) $ such that one of the functions $\varphi_n $ is a
constant, say
\begin{equation}
\label{e1} \varphi_0 (z) \equiv 1.
\end{equation}
Let $B(r) = \{z=(z_1, \ldots, z_d)\in \mathbb{C}^d: \; \sum_1^d
|z_k|^2 \leq r^2\}. $ It is enough to prove that for every $r>0 $
there is $R>r $ such that
\begin{equation}
\label{e3} \text{if} \;\;  f=\sum_0^\infty f_n \varphi_n, \;\;
\text{then} \;\; \sum_0^\infty |f_n | |\varphi_n |_{B(r)} \; \leq \;
|f|_{B(R)} \quad \forall f \in H(\mathbb{C}^d).
\end{equation}

One can easily see that the system
\begin{equation}
\label{e4}  \psi_0 (z) \equiv 1, \quad \psi_n (z) = \varphi_n (z) -
\varphi_n (0), \quad n\in \mathbb{N},
\end{equation}
is also a basis in $H(\mathbb{C}^d).$ Moreover, if $f =
\sum_{n=0}^\infty f_n \varphi_n, $ then we have $f(0) =
\sum_{n=0}^\infty f_n \varphi_n (0),$ which implies that $f= f(0) +
\sum_{n=1}^\infty f_n \psi_n.$

First we show that for every $r>0$ there is $R>r $ such that
\begin{equation}
\label{e5} \text{if} \;\;  f= f(0) +  \sum_1^\infty f_n \psi_n, \;
\text{then} \; |f(0)| + \sum_1^\infty |f_n | |\psi_n |_{B(r)} \; \leq
\; |f|_{B(R)} \;\; \forall f \in H(\mathbb{C}^d).
\end{equation}
Fix $r>0 $ and a function $f \in H(\mathbb{C}^d).$ We may assume
without loss of generality that $f(0)\geq 0$ (otherwise one may
multiply $f$ by $ |f(0)|/f(0)$). By Proposition~\ref{abs}, there are
$C > 0 $ and  $r_1 > r $ (which do not depend on $f$) such that
$$
f(0) + \sum_1^\infty |f_n | |\psi_n |_{B(r)} \; \leq \; f(0) + C
|f-f(0)|_{B(r_1)}.
$$
Now, for any $r_2 > r_1,$ the Borel - Carath\'eodory theorem (see
\cite{T}) says that
$$|f-f(0)|_{B(r_1)} \leq \frac{2r_1}{r_2 - r_1} \sup_{B(r_2)} (Re f(z) - f(0) ).
$$
Let $r_2 = (2C+1) r_1.$ Then $\frac{2Cr_1}{r_2 - r_1} = 1$, so we
obtain
$$
f(0) + \sum_1^\infty |f_n | |\psi_n |_{B(r)} \leq f(0) +
\frac{2Cr_1}{r_2-r_1} \sup_{B(r_2)} (Re f(z) - f(0) ) \leq
\sup_{B(r_2)} Re f(z),  $$ i.e., (\ref{e5}) holds with $ R=(2C+1)
r_1.$

Next we prove (\ref{e3}). Since $\varphi_0 \equiv 1,$ we have
$$
f(0)= f_0 + \sum_{n=1}^\infty f_n \varphi_n (0) \quad  \Rightarrow
\quad |f_0| \leq |f(0)| + \sum_{n=1}^\infty |f_n| |\varphi_n (0)|,$$
so it follows that
$$
\sum_0^\infty |f_n | |\varphi_n |_{B(r)} \leq |f_0| + \sum_1^\infty
|f_n | |\varphi_n -\varphi_n (0) |_{B(r)} +\sum_{n=1}^\infty |f_n|
|\varphi_n (0)|
$$
$$
\leq |f(0)| + \sum_1^\infty |f_n | |\varphi_n -\varphi_n (0) |_{B(r)}
+ 2\sum_{n=1}^\infty |f_n| |\varphi_n (0)|.
$$
Applying the Schwartz Lemma, we obtain
\begin{equation}
\label{e8} \sum_0^\infty |f_n | |\varphi_n |_{B(r)} \leq |f(0)| +
\frac{1}{3}\sum_1^\infty |f_n | |\varphi_n - \varphi_n (0) |_{B(3r)}
+ 2\sum_{n=1}^\infty |f_n| |\varphi_n (0)|.
\end{equation}
On the other hand, by Lemma~\ref{lem1}, there is $\tilde{r} \geq 3r $
such that
$$
|\varphi_n (0)|/|\varphi_n |_{B(\tilde{r})} \leq    \frac{1}{4},
\quad \forall n \in \mathbb{N}.
$$
Therefore, we have
\begin{equation}
\label{e9} |\varphi_n - \varphi_n (0) |_{B(\tilde{r})} \geq
|\varphi_n|_{B(\tilde{r})} -|\varphi_n (0)| \geq 3|\varphi_n (0)|,
\quad \forall n \in \mathbb{N}.
\end{equation}
Since $\tilde{r} \geq 3r, $  (\ref{e8}) and (\ref{e9}) imply that
$$
\sum_0^\infty |f_n | |\varphi_n |_{B(r)} \leq |f(0)| + \sum_1^\infty
|f_n | |\varphi_n -\varphi_n (0) |_{B(\tilde{r})}.
$$
Hence, from (\ref{e4}) and (\ref{e5}) it follows that (\ref{e3})
holds. This completes the proof.

\end{proof}

\begin{Remark}
\label{rem4} Suppose that $(\varphi_n)_{n=0}^\infty, \; \varphi_0
\equiv 1, $ is a basis in the space of entire functions
$H(\mathbb{C}^d) $ and $K \subset K_1 $ are compact sets such that
for every entire function $h= \sum_{n=0}^\infty h_n \varphi_n $ we
have
\begin{equation}
\label{r1} \sum_{n=0}^\infty |h_n| |\varphi_n|_K \leq |h|_{K_1}.
\end{equation}
If a domain $G \supset K_1 $ has the property that $(\varphi_n)$ is a
basis in the space $H(G)$ of holomorphic functions on $G,$ then for
every bounded function $f \in H(G)$ with $f= \sum f_n \varphi_n $ we
have
\begin{equation}
\label{r2} \sum_{n=0}^\infty |f_n| |\varphi_n|_K \leq \sup_{G}
|f(z)|.
\end{equation}
\end{Remark}

\begin{proof}
If $f$ is a bounded holomorphic function on $G,$ then $$f=
\sum_{n=0}^\infty f_n \varphi_n,$$ where the series converges
uniformly on every compact subset of G, so in particular it converges
uniformly on $K_1.$

Let $(S_m)$ be the sequence of partial sums of the expansion of $f.$
Then (\ref{r1}) holds for $S_m,$ i.e.,
$$
\sum_{n=0}^m |f_n| |\varphi_n |_K  \leq \sup_{K_1} |S_m (z)|,
$$
and  $S_m \to f$ uniformly on $K_1.$ Therefore, it follows that
$$
\sum_{n=0}^\infty |f_n| |\varphi_n |_K  \leq \sup_{K_1} |f(z)|  \leq
\sup_{G} |f(z)|,
$$
which completes the proof.
\end{proof}

Some of the bases encountered for the space of entire functions are
also common bases for the spaces of analytic functions corresponding
to an exhaustive one parameter family of sub domains $\mathcal{D}_r$
(usually sub level domains of an associated plurisubharmonic
function). So if such a basis, say $(\varphi_n),$  has $GBP$ then for
every compact set $K \subset \mathbb{C}^d$, there exits a domain
$\mathcal{D}_r$ such that (\ref{r2}) holds with $G=\mathcal{D}_r, $
i.e.,
\begin{equation}
\label{r4}   \forall K \subset \subset \mathbb{C}^d \;\; \exists r:
\quad \sum_{n=0}^\infty |c_n| |\varphi_n|_K \leq \sup_{D_r} |f(z)|.
\end{equation}

Recently, a partial case of this situation have been considered by P.
Lass\`{e}re and E. Mazzilli \cite{LM1}. For a fixed continuum $K
\subset \mathbb{C},$ they studied the Bohr phenomenon for Faber
polynomial basis $(F_{K,n})_{n=0}^\infty$ in $H(\mathbb{C})$ (which
is a common basis for $H(\mathcal{D}_r),$ where $\mathcal{D}_r$  are
the sub level domains of the Green function of the complement of $K$
with pole at infinity). They established Bohr's phenomenon in the
form (\ref{r4}), i.e., there exists $r>0$ such that, for every
bounded function $f \in H(\mathcal{D}_r)$ with $f= \sum c_n F_{K,n},$
$$
\sum_{n=0}^\infty |c_n| |F_{K,n}|_K  \leq \sup_{D_r} |f(z)|
$$
holds.  P. Lass\`{e}re and E. Mazzilli called the infimum $R_0$ of
such $r$ {\em Bohr radius} associated to the family $(K,
\mathcal{D}_r, (F_{K,n})_{n=0}^\infty)$  (see also \cite{LM2}).

The definition of P. Lass\`{e}re and E. Mazzilli could be extended to
the general situation mentioned above. Namely, if
$(\varphi_n)_{n=0}^\infty, \; \varphi_0 \equiv 1, $  is a basis in
the space of entire functions $H(\mathbb{C}^d)$ such that it is a
common basis for the spaces $H(\mathcal{D}_r),$  where
$(\mathcal{D}_r)$ is an exhaustive family of domains in
$\mathbb{C}^d,$ then by Theorem~\ref{thm1} and Remark~\ref{rem4} it
follows that (\ref{r4}) holds. So, following P. Lass\`{e}re and E.
Mazzilli, we may introduce {\em Bohr radius} associated to the family
$(K, \mathcal{D}_r, (\varphi_n)_{n=0}^\infty),$ where $K$ is a fixed
compact set, as the infimum $R_0$ of all $r$ such that the inequality
in (\ref{r4}) holds.

\section{Stein manifolds with Liouville Property}
In this section we generalize Theorem~\ref{thm1} to  Stein manifolds
$\mathcal{M} $ without nontrivial global bounded analytic functions.

\begin{Definition}
\label{LP} We say that a Stein manifold $\mathcal{M} $ has Liouville
Property if every bounded analytic function on $\mathcal{M} $ is a
constant.
\end{Definition}

Next we introduce another property that is crucial for our
considerations.
\begin{Definition}
\label{SP} We say that a Stein manifold $\mathcal{M} $ has Schwarz
Property if for every  $z_0 \in \mathcal{M},$ compact $K \ni z_0 $
and $\delta >0 $  there is a compact $K_1 $ with
\begin{equation}
\label{sch}     |f-f(z_0)|_K \leq \delta \, |f-f(z_0)|_{K_1} \quad
\forall f \in H (K_1),
\end{equation}
where $H(K_1)$ denotes the set of functions analytic in a
neighborhood of $K_1.$
\end{Definition}

 It turns out that Schwarz Property and Liouville Property are
equivalent. In order to prove this fact we introduce yet another
property.

Let $\mathcal{M} $ be a Stein manifold. Fix a point  $z_0 \in
\mathcal{M},$ and choose an exhaustion $(D_n)_{n\in\mathbb{N}}$ of
$\mathcal{M} $ consisting of open relatively compact sets so that
\begin{equation}
\label{dn} z_0 \in D_1, \quad \overline{D_n} \subset D_{n+1} \quad
\forall n \in \mathbb{N}.
\end{equation}
For each $n,$  consider the family of functions
\begin{equation}
\label{fn} \mathcal{F}_n = \{f \in H(D_n): \; f(z_0)=0, \; \sup_{D_n}
|f(z)|\leq 1\},
\end{equation}
and set
\begin{equation}
\label{gamma} \gamma_n (z) = \sup \{|f(z)|: \; f \in \mathcal{F}_n\},
\quad z\in D_n, \quad n \in \mathbb{N}.
\end{equation}
The functions $\gamma_n (z) $ are continuous. Indeed, fix $n\in
\mathbb{N},$  $w \in D_n $ and $\varepsilon >0. $  The family of
functions $\mathcal{F}_n $ is compact, and therefore, equicontinuous
on $D_n.$   Therefore, there is $\delta >0$  such that if $z\in D_n $
and  $dist (z, w) < \delta $ then
$$\left | \,|f(z)|-|f(w)|\, \right |\leq |f(z)-f(w)| < \varepsilon
\quad \forall f \in \mathcal{F}_n.$$ Thus we obtain, for every fixed
$z \in D_n $ with $dist (z, w) < \delta, $
$$
|f(z)| \leq |f(w)| + \varepsilon \quad \forall f \in \mathcal{F}_n
\Rightarrow |f(z)| \leq \gamma_n (w) + \varepsilon \quad \forall f
\in \mathcal{F}_n,
$$
which implies that
$$
\gamma_n (z) \leq \gamma_n (w) + \varepsilon.
$$
On the other hand, we have also that if $dist (z, w) < \delta $ then
$$
|f(w)| \leq |f(z)| + \varepsilon \quad \forall f \in \mathcal{F}_n
\Rightarrow |f(w)| \leq \gamma_n (z) + \varepsilon \quad \forall f
\in \mathcal{F}_n,
$$
which implies that
$$
\gamma_n (w) \leq \gamma_n (z) + \varepsilon.
$$
Hence we obtain that $|\gamma_n (z)- \gamma_n (w)|< \varepsilon $ if
$z \in D_n $ and  $dist (z, w) < \delta, $  i.e., $\gamma_n $ is
continuous at $w.$

Let us note that
\begin{equation}
\label{gg1} \gamma_n (z) \geq \gamma_{n+1} (z)  \quad \forall z\in
D_n, \;\;  n\in \mathbb{N}.
\end{equation}

\begin{Theorem}
\label{thm} Let $\mathcal{M}$ be a Stein manifold. The following
conditions on $\mathcal{M}$ are equivalent:

(i) $\mathcal{M}$ has Liouville Property;

(ii) $\mathcal{M}$ has Schwarz Property;

(iii) For every exhaustion $(D_n)$ of $\mathcal{M}$ of the kind
(\ref{dn}) the corresponding functions (\ref{gamma}) satisfy $\lim
\gamma_n (z) = 0.$
\end{Theorem}

\begin{proof}
First we show that (i) implies (iii). Assume the contrary, that (i)
holds but (iii) fails. Then there is an exhaustion $(D_n)$ of
$\mathcal{M}$ of the kind (\ref{dn}) and a point $w \in \mathcal{M} $
such that the corresponding functions (\ref{gamma}) satisfy $\lim
\gamma_n (w) = c >0,$ so $\gamma_n (w) > c/2 $ for large enough $n,$
say $n\geq n_0.$ In view of the definition of $\gamma_n (w) $ (see
\ref{fn}) and (\ref{gamma})),  there is a sequence of functions
$(f_n)_{n\geq n_0}$  such that $f_n \in \mathcal{F}_n$ and $|f_n (w)|
> c/2.$ For every fixed $k$ the sequence $(f_n)$ is uniformly bounded
by $1 $ on $\overline{D}_k $ for $n>k.$ Now the Montel property and a
standard diagonal argument show that there is a subsequence
$f_{n_\nu}$ such that $|f_{n_\nu}(z)| \leq 1$ for $z\in D_{n_\nu},$
$f_{n_\nu} (z_0)=0, \; f_{n_\nu} (w)>c/2,$ and a tail of $f_{n_\nu}$
converges uniformly on every compact subset of $\mathcal{M}.$  The
limit function $f(z) = \lim f_{n_\nu} (z)$  will be bounded by $1$
and $f(z_0)=0, \; |f(w)|>c/2>0. $ The existence of such a function
contradicts (i). Hence $(i) \Rightarrow (iii). $

Next we prove that $(iii) \Rightarrow (ii).$  Indeed, fix a compact
$K$ and choose $m\in \mathbb{N}$ so that $D_m \supset K.$ By Dini's
theorem, $(\gamma_n (z))_{n > m} $ converges uniformly to $0$ on
$\overline{D}_m.$ Therefore, for every $\delta >0 $ there is $n_1 >m$
such that $\gamma_{n_1} (z) < \delta   $ for $z \in D_{m}.$  But this
means that (\ref{sch}) holds with $K_1 = \overline{D}_{n_1}.$

Finally, we show that (ii) implies (i). Let $f$ be a bounded analytic
function on $\mathcal{M}, $ say $|f(z)| \leq C.$ Fix $z\in
\mathcal{M}, \; z\neq z_0.$  By Schwartz Property, it follows that
for every $\delta >0 $ we have $|f(z) - f(z_0)| < 2C \,\delta, $ so
$f(z) = f(z_0).$  Hence $f = const, $ which proves that (i) holds.
\end{proof}

Next we consider a  generalization of the classical
Borel-Carath\'eodory theorem.

\begin{Theorem}
\label{BK} Let $\mathcal{M} $ be a Stein manifold with Liouville
Property. Then for every $z_0 \in \mathcal{M},$   compact $K \ni z_0$
and $\varepsilon > 0 $ there is a compact $K_1 $ with
\begin{equation}
\label{bk}     |f-f(z_0)|_K \leq \varepsilon  \, \sup_{K_1} Re\,(f(z)
-f(z_0)) \quad \forall f \in H (K_1).
\end{equation}
\end{Theorem}

\begin{proof}
Fix $z_0 \in \mathcal{M},$ compact $K \ni z_0$ and $\varepsilon > 0.
$ Let
$$
\delta= \varepsilon/(2+\varepsilon),
$$
and let  $K_1\subset \mathcal{M} $ be a compact  such that
(\ref{sch}) holds.

Fix a non-constant function $f \in H (K_1)$ such that $f(z_0) =0,$
and set
\begin{equation}
\label{g1} A=\sup_{K_1} Re\,f(z).
\end{equation}

Set
$$
g (z) = f(z)/(2A-f (z));
$$
then $g \in H(K_1), \; g(z_0)=0,$ and in view of (\ref{g1}) it
follows that
$$
  |g (z)| \leq 1 \quad \text{for} \quad z\in K_1.
$$
Therefore, by (\ref{sch}),
$$
|g (z)| = \left |\frac{f(z)}{2A-f(z)}\right | \leq \delta \quad
\text{if} \;\;z\in K,
$$
so $ |f(z)| \leq \delta|2 A- f(z)|\leq 2\delta A + \delta |f(z|, $
and it follows that
$$
|f(z)| \leq  \frac{2\delta}{1-\delta}\,A = \varepsilon\,A, \quad z\in
K.
$$
Hence, we obtain
$$
|f(z)| \leq \varepsilon \sup_{K_1} Re\,(f(z)) \quad \forall z \in K,
$$
which completes the proof.
\end{proof}

\begin{Theorem}
\label{thm2} Let $\mathcal{M}$ be a Stein manifold with Liouville
Property. Suppose $(\varphi_n)_{n=0}^\infty$ is a basis in
$H(\mathcal{M})$ such that
$$
\varphi_0 (z) \equiv 1.
$$
Then for every compact $K\subset \mathcal{M}$ there is a compact $K_1
$ such that
\begin{equation}
\label{g2} \text{if} \;\;  f=\sum_0^\infty f_n \varphi_n, \;\;
\text{then} \;\; \sum_0^\infty |f_n | |\varphi_n |_K \; \leq \;
|f|_{K_1} \quad \forall f \in H(\mathcal{M}).
\end{equation}

\end{Theorem}

\begin{proof}
Fix a compact $K$ and a point $z_0 \in K. $  The system
\begin{equation}
\label{g4}  \psi_0 (z) \equiv 1, \quad \psi_n (z) = \varphi_n (z) -
\varphi_n (z_0), \quad n\in \mathbb{N}.
\end{equation}
is also a basis in $H(\mathcal{M}).$ Moreover, if $f=f_0 +
\sum_{n=1}^\infty f_n \varphi_n, $  then we have $f(z_0) = f_0
+\sum_{n=1}^\infty f_n \varphi_n (z_0),$ so it follows that
$$
f= f(z_0) +  \sum_1^\infty f_n \psi_n.
$$

First we are going to show that there is a compact $K_1 $ such that
\begin{equation}
\label{g5} \text{if} \;\;  f= f(z_0) +  \sum_1^\infty f_n \psi_n,
\;\; \text{then} \;\; |f(z_0)| + \sum_1^\infty |f_n | |\psi_n |_{K}
\; \leq \; |f|_{K_1} \quad \forall f \in H(\mathcal{M}).
\end{equation}
Fix  a function $f \in H(\mathcal{M}).$ We may assume without loss of
generality that $f(z_0)\geq 0$ (otherwise one may multiply $f$ by $
|f(z_0)|/f(z_0)$). By Proposition~\ref{abs}, there are $C \geq 1 $
and a compact $\tilde{K} \supset K $ (which do not depend on $f$)
such that
$$
\sum_1^\infty |f_n | |\psi_n |_{K} \; \leq \;  C
|f-f(z_0)|_{\tilde{K}}.
$$
By Theorem~\ref{BK}, with $\varepsilon =1/C, $ there is a compact
$K_1 \supset \tilde{K} $ such that (\ref{bk}) holds with
$K=\tilde{K}.$ Since $\sup_{K_1} Re\,(f(z)-f(z_0)) \leq |f|_{K_1}-
f(z_0),$ we obtain
$$
f(z_0) + \sum_1^\infty |f_n | |\psi_n |_{K} \; \leq \; f(z_0) + C
\varepsilon \,(|f|_{K_1}- f(z_0)) = |f|_{K_1},
$$
i.e., (\ref{g5}) holds.

Lemma~\ref{lem1} can be readily generalized to the case of Stein
manifolds with Liouville Property, so the same argument as in the
proof of Theorem~\ref{thm1} shows that (\ref{g5}) implies (\ref{g2}).
This completes the proof.
\end{proof}

\begin{Corollary}
Let $\mathcal{M}$ be a Stein manifold with Liouville Property. A
basis $(\varphi_n (z))_{n=0}^\infty$ in $H(\mathcal{M})$ has $GBP$ if
and only if one of the functions $\varphi_n$ is a constant.
\end{Corollary}
Indeed, the necessity part follows from Proposition 3.1 in
\cite{AAD01}, or as in Theorem~\ref{thm1}.\bigskip

Obviously, the class of manifolds with  Liouville Property include
$\mathbb{C}^d, \; d=1,2, \ldots. $ In order to give more examples let
us recall that a complex manifold is called {\em parabolic} if every
bounded from above plurisubharmonic function reduces to a constant.
In view of the fact that the modulus of analytic functions are
plurisubharmonic, parabolic manifolds possess  Liouville Property.
Affine algebraic manifolds, $\mathbb{C}^{d} \setminus Z(F) $  with
$Z(F)$ being the set of zeros of an entire function $F$, parabolic
Riemann surfaces \cite{Alf}, are parabolic. For more examples see
\cite{AS}.  However not every complex manifold with  Liouville
Property is parabolic -- see \cite{AS} for a simple example.

Let us note that the class of parabolic manifolds is remarkable by
the fact that the space of global analytic functions on a parabolic
manifold admits a basis (see \cite{AKT}). For a general complex
manifold $\mathcal{M}$ it is an open question as to whether or not
$H(\mathcal{M})$ possesses a basis.


\begin{thebibliography}{99}


\bibitem{Aiz00}  L. Aizenberg,  Multidimensional analogues
of Bohr's theorem on power series, Proc. Amer. Math. Soc. {\bf 128}
(2000), 1147--1155.

\bibitem{Aiz05} L. Aizenberg, Generalization of Carath\'{e}odory's
inequality and the Bohr radius for multidimensional power series.
Selected topics in complex analysis, Oper. Theory Adv. Appl., {\bf
158}, Birkh\"{a}user, Basel, (2005), 87--94.



\bibitem{AAD00} L. Aizenberg, A. Aytuna and P. Djakov,
An abstract approach to Bohr's phenomenon. Proc. Amer. Math. Soc.
{\bf 128} (2000), no. 9, 2611--2619.


\bibitem{AAD01} L. Aizenberg, A. Aytuna and P. Djakov,
Generalization of a theorem of Bohr for bases in spaces of
holomorphic functions of several complex variables. J. Math. Anal.
Appl. {\bf 258} (2001), no. 2, 429--447.



\bibitem{Alf}  L. Ahlfors and L. Sario, {\em Riemann surfaces},
Princeton University Press, Princeton, New Jersey, 1960.

\bibitem{AKT} A. Aytuna, J. Krone, and T. Terzioglu, Complemented
infinite type power series subspaces of nuclear Fr\'{e}chet spaces.
Math. Ann. 283 (1989), 193--202.


\bibitem{AS} A. Aytuna and A. Sadullaev, Parabolic Stein manifolds.
Preprint arXiv:1112.1626v1 [math.CV] 2011.


\bibitem{BDK} C. Beneteau, A. Dahlner and D. Khavinson, Remarks on the
Bohr phenomenon, Comput. Methods Funct. Theory {\bf 4} no.1 (2004),
1--19.

\bibitem{BK} H. P. Boas, D. Khavinson, {\em Bohr's power series
theorem in several variables}, Proc. Amer. Math. Soc. {\bf 125}
(1997), 2975--2979.

\bibitem{B} H. Bohr, {\em A theorem concerning power series},
Proc. London Math. Soc. (2) {\bf 13} (1914), 1--5.


\bibitem{BB} E. Bombieri and J. Bourgain,  A remark on Bohr's inequality.
Int. Math.Res. Not. 80 (2004), 4307--4330.

\bibitem{DGM} A. Defant, D. Garclnot a and M. Maestre, Bohr's power
series theorem and local Banach space theory. J. Reine Angew. Math.
557 (2003), 173--197.


\bibitem{DP} A. Defant, C. Prengel, Christopher Harald Bohr meets Stefan
Banach. Methods in Banach space theory, London Math. Soc. Lecture
Note Ser., {\bf 337}, Cambridge Univ. Press, Cambridge, (2006),
317--339.

\bibitem{DFOOS} A. Defant, L. Frerick, J. Ortega-Cerd\'{a},
M. Ouna\"{\i}es and K. Seip, The Bohnenblust-Hille inequality for
homogeneous polynomials is hypercontractive. Annals of Mathematics
{\bf 174} (2011), Issue 1, 485--497.



\bibitem{DT} S. Dineen and  R. M. Timoney,  Absolute bases, tensor
products and a theorem of Bohr Studia Math. {\bf 94} (1989),
227--234.

\bibitem{Dix} P. G. Dixon, Banach algebra satisfying the non-unital von
Neumann inequality. Bull. London Math. Soc. {\bf 27} 1995, 359--362.


\bibitem{DR} P. Djakov and M. Ramanujan, A remark on Bohr's theorem and
its generalizations. J.Anal. {\bf 8} (2000), 65--77.

\bibitem{DM60} A. Dynin and B. Mityagin,  Criterion for
nuclearity in terms of approximative dimension, Bull. Pol. Acad.
Sci., Math., {\bf 8} (1960), 535--540.


\bibitem{KS} H. T. Kaptanoglu and N. Sadik, Bohr Radii of elliptic functions.
Russian Journal of Mathematical Physics, 12-3 (2005), 365--368.


\bibitem{LM1} P. Lassere and E. Mazzilli,
Bohr's phenomenon on a regular condensator in the complex plane,
Comput. Methods Funct. Theory  12, no.1 (2012), 31--43.

\bibitem{LM2} P. Lassere and E. Mazzilli,
The Bohr Radius for an Elliptic Condenser, arXiv:1109.4511.

\bibitem{M}  B. S. Mityagin,  Approximative dimension
and bases in nuclear spaces, Uspekhi Mat. Nauk {\bf 16}, 63--132
(1961), (Russian),  English Transl.-Russian Math. Surveys {\bf 16},
59--127 (1961).


\bibitem{MV} R. Meise and D. Vogt,  {\em Introduction to Functional Analysis},
Clarendon Press, Oxford 1997.

\bibitem{PPS} V. Paulsen, G. Popescu and D. Singh, On Bohr's inequality, Proc.
London Math. Soc. {\bf 85} (2002), no.2, 493--512.

\bibitem{G}   G. Roos,  H. Bohr's theorem for bounded symmetric
domains,  arXiv:0812.4815.

\bibitem{T} E. C. Titchmarsh, {\em The theory of functions},
Oxford University Press, 1939.


\end{thebibliography}
\end{document}